\documentclass[12pt, twoside]{article}

\usepackage[latin2]{inputenc}
\usepackage[T1]{fontenc}
\usepackage{euscript}
\usepackage{amsmath}
\usepackage{amsthm}
\usepackage{amssymb}
\usepackage[matrix,arrow]{xy}
\usepackage{enumerate}
\usepackage{times}

\usepackage[bookmarksnumbered,bookmarksopen,unicode]{hyperref}
\def\titlerunning#1{\gdef\titrun{#1}}
\makeatletter
\def\author#1{\gdef\autrun{\def\and{\unskip, }#1}\gdef\@author{#1}}
\def\address#1{{\def\and{\\\hspace*{18pt}}\renewcommand{\thefootnote}{}%
\footnote {#1}}%
\markboth{\autrun}{\titrun}}
\makeatother
\def\email#1{e-mail: #1}
\def\subjclass#1{{\renewcommand{\thefootnote}{}%
\footnote{\emph{Mathematics Subject Classification (2010):} #1}}}
\def\keywords#1{\par\medskip
\noindent\textbf{Keywords.} #1}

\numberwithin{equation}{section}
\numberwithin{equation}{subsection}

\theoremstyle{plain}
\newtheorem{theorem}[equation]{Theorem}

\newtheorem{proposition}[equation]{Proposition}
\newtheorem{corollary}[equation]{Corollary}

\theoremstyle{definition}
\newtheorem{definition}[equation]{Definition}

\newtheorem{bekezd}[equation]{}

\newcommand{\cO}{\mathcal{O}}
\newcommand{\cH}{\mathcal{H}}

\newcommand{\cV}{\mathcal{V}}
\newcommand{\cZ}{\mathcal{Z}}
\newcommand{\cF}{\mathcal{F}}
\newcommand{\cP}{\mathcal{P}}

\newcommand{\cS}{\mathcal{S}}
\newcommand{\cE}{\mathcal{E}}

\newcommand{\calm}{{\mathcal M}}
\newcommand{\calj}{\mathcal{V}}

\newcommand{\calQ}{\mathcal{Q}}\newcommand{\calP}{\mathcal{P}}
\newcommand{\calF}{\mathcal{F}}\newcommand{\calU}{{\mathcal U}}
\newcommand{\q}{w}
\newcommand{\bz}{{\bf z}}

\DeclareMathOperator{\rank}{{\rm rank}}

\DeclareMathOperator{\Hom}{{\rm Hom}}
\newcommand{\et}{\mathcal{T}}

\newcommand{\setQ}{\mathbb{Q}}\newcommand{\Q}{\mathbb{Q}}
\newcommand{\R}{\mathbb{R}}
\newcommand{\setZ}{\mathbb{Z}}\newcommand{\Z}{\mathbb{Z}}

\def\bH{\mathbb H}

\newcommand{\hh}{\mathfrak{h}}
\newcommand{\bt}{{\bf t}}
\newcommand{\bx}{{\bf x}}

\newcommand{\frs}{\mathfrak{s}}\newcommand{\frsw}{\mathfrak{sw}}

\providecommand{\coloneqq}{\mathrel{:=}}

\newcommand{\labelpar}{\label}

\begin{document}

\baselineskip=17pt

\titlerunning{The SW-invariant of plumbed 3--manifolds}

\title{The Seiberg--Witten invariants of negative definite plumbed 3--manifolds}

\author{Andr\'as N\'emethi}

\date{}

\maketitle

\address{R\'enyi Institute of Mathematics,  1053 Budapest,   Re\'altanoda u. 13--15,
 Hungary;
\email{nemethi@renyi.hu}}


\begin{abstract}
Assume that $\Gamma$ is a connected negative definite plumbing graph, and that
the associated plumbed 3--manifold  $M$ is a rational homology sphere.
We provide two new combinatorial formulae for the Seiberg--Witten invariant of $M$.
The first one is the constant term of a `multivariable Hilbert polynomial',
it reflects in a  conceptual  way the structure of the
graph  $\Gamma$, and emphasizes the subtle parallelism between these topological
invariants and the analytic invariants of normal surface singularities.
The second formula realizes the Seiberg--Witten invariant as the normalized Euler characteristic
of the lattice cohomology associated with $\Gamma$,
supporting the conjectural connections between the Seiberg--Witten Floer
homology, or the  Heegaard--Floer homology, and the lattice cohomology.
\end{abstract}

\section{Introduction}\label{sec:introduction}

Let $\Gamma$ be a connected negative definite plumbing graph with vertices $\cV$.
We assume that it is a tree, and all the plumbed surfaces have genus zero.
Hence, the associated oriented plumbed 3--manifold $M=M(\Gamma)$ is a rational homology sphere.
We denote by  $\frsw_{\sigma}(M)$
the Seiberg--Witten invariants of $M$
indexed by the $spin^c$--structures $\sigma$ of $M$. Although in the recent years several combinatorial
formulae were established for them, their computation is still very difficult and involved.
E.g., in \cite{Nico5} it is proved that they are equivalent with Turaev's torsion
normalized by the Casson--Walker invariant (a result based on the surgery
formulas of \cite{MW}). In terms of $\Gamma$,
 a combinatorial formula for the Casson--Walker invariant
 can be deduced from Lescop's book \cite{Lescop},
while  the Turaev's torsion  is determined in \cite{NN1}.
Nevertheless, this expression of the torsion is based on a Dedekind--Fourier sum, which,
in most of the  particular cases, is hard to determine.

For some special graphs, for the computation of the Seiberg--Witten invariant
 one can use results of the Heegaard--Floer homology too, especially surgery formulae, see e.g.
\cite{OSzP,NOSZ,Rus}. Moreover, for arbitrary graphs,
\cite{BN} provides a different type of surgery formula (motivated by singularity theory).
 In fact, in this note we  rely exactly on this surgery formula from  \cite{BN}.

Our goal is to provide two new combinatorial formulae for
$\frsw_{\sigma}(M)$.
One of them uses qualitative properties of the coefficients of a combinatorial zeta function associated with
$\Gamma$, the other is the normalized Euler characteristic of the lattice cohomology
of  $\Gamma$ (introduced in \cite{NLC}).
Both formulae reflect in the most conceptual and optimal way the structure of the
graph  $\Gamma$, and emphasizes the subtle parallelism between these topological
invariants and the analytic invariants of normal surface singularities.
The main aim is to establish the identity (and unity) of these three objects: Seiberg--Witten invariant,
periodic constant of the zeta function, and the Euler characteristic of the lattice cohomology.

In order to formulate these correspondences,
let us consider the plumbed 4--manifold $\widetilde{X}$ associated with
$\Gamma$. Its second
homology $L$ is freely generated by the 2--spheres $\{E_v\}_{v\in\cV}$, and its
second cohomology $L'$
by the (anti)dual classes $\{E^*_v\}_{v\in\cV}$; the intersection
form $I=(\,,\,)$ embeds $L$ into $L'$, for details see (\ref{ss:11}).
(Equivalently, $L$ is the combinatorial lattice with intersection form $I$ associated with
$\Gamma$, and $L'$ is its dual lattice, and both are endowed with their natural bases). Set $x^2:=(x,x)$.

Let $K\in L'$ be the canonical class (see (\ref{eq:adjun})),
$\widetilde{\sigma}_{can}$ the canonical
$spin^c$--structure on $\widetilde{X}$ (with $c_1(\widetilde{\sigma}_{can})=-K$) 
and $\sigma_{can}\in \mathrm{Spin}^c(M)$ its restriction on $M$
(see (\ref{ss:SW})).

 Consider the  multi-variable
Taylor expansion $Z(\bt)=\sum p_{l'}\bt^{l'}$ at the  origin of
\begin{equation}\label{eq:INTR}\prod_{v\in \cV} (1-\bt^{E^*_v})^{\delta_v-2},\end{equation}
where for any $l'=\sum _vl_vE_v\in L'$ we write
$\bt^{l'}=\prod_vt_v^{l_v}$, and $\delta_v$ is the valency of $v$.
 This lives in $\setZ[[L']]$, the submodule of
formal power series $\setZ[[\bt^{\pm 1/d}]]$ in variables $\{t_v^{\pm 1/d}\}_v$, where $d=\det(-I)$.
The first identity is the following.

\vspace{2mm}

\noindent {\bf Theorem A.} \ {\it Fix some $l'\in L'$.
Assume that for any  $v\in\cV$ the  $E^*_v$--coordinate  of $l'$
is larger than or equal to $-(E_v^2+1)$. Then   the sum
$$\sum_{l\in L,\, l\not\geq 0}p_{l'+l }$$
equals  a multivariable quadratic function on $l'$, namely
\begin{equation}\label{eq:INNSW}-\frac{(K+2l')^2+|\cV|}{8}-\frs_{[-l']},
\end{equation}
where  the constant $\frs_{[-l']}$ depends only on the class
$[-l']$ of $-l'$ in $L'/L=H^2(M,\setZ)$.

Moreover, if $*$ denotes the (torsor) action of $L'/L$ on $\mathrm{Spin}^c(M)$, one has
$$\frs_{[l']}=\frsw_{[l']*\sigma_{can}}(M).$$
In particular, the normalized Seiberg--Witten invariant appears as the constant term of
the `combinatorial multivariable Hilbert polynomial' (\ref{eq:INNSW}). }

\vspace{2mm}

For the second identity let us consider the lattice cohomology
$\bH^*(\Gamma)$ associated with $\Gamma$.
It depends only on $M$, and it has a natural direct sum decomposition indexed by
$\sigma\in \mathrm{Spin}^c(M)$, namely
$\bH^*(\Gamma)=\oplus_{\sigma}\bH^*(\Gamma,\sigma)$. Let $eu(\bH^*(\Gamma,\sigma))$
be the normalized Euler characteristic of the corresponding summand;
for more details see (\ref{ss:LatCoh}).
Then one has

\vspace{2mm}

\noindent {\bf Theorem B.} \ {\it For any $\Gamma$ and $\sigma $ as above }
 $$-eu(\bH^*(\Gamma,\sigma))=
\frsw_{\sigma}(M(\Gamma)).$$

\vspace{2mm}

In fact, Theorem A was motivated by a similar formula
valid for equivariant geometric genera of normal surface singularities,
cf. \cite{coho3}; this is explained in (\ref{FM}). The combinatorial
quadratic Hilbert-polynomial type behaviour (\ref{eq:INNSW}) is proved in (\ref{th:1}).
It uses essentially the preparatory part of subsection (\ref{ss:LatCoh}), where we review and prove some
statements about lattice cohomology, and we identify the two combinatorial objects as
\begin{equation}\label{E}\frs_{[-l']}=-eu(\bH^*(\Gamma, [K+2l'])).
\end{equation} The second part of Theorem A rely  on
a surgery formula for the constant term $\frs$,
which fits perfectly with the surgery formula proved for the Seiberg--Witten invariant
$\frsw$ in \cite{BN}. This allows us to prove in subsection (\ref{s:SW}) the identity
 $\frs=\frsw$ by induction on $|\cV|$.

The  surgery formula  involves in a crucial way the `periodic constant' of a series
introduced in \cite{NO1,Opg}, see (\ref{ss:pc}).
In fact, via Theorem A, the Seiberg--Witten invariants can be interpreted
as the `multivariable periodic constants' of the series $Z(\bt)$.

The series $Z(\bt)$ was used in several articles studying invariants of surface singularities
\cite{CDG,CDGEq,CHR,CDGb,coho3}. Theorem A puts the results of these articles in a new light. Indeed,
as a consequence of the present work,
the identity $Z(\bt)$ with the analytic invariant $\cP(\bt)$, (see  (\ref{FM}) for its definition),
in some articles called Campillo--Delgado--Gusein-Zade type identity, implies automatically the
Seiberg--Witten Invariant Conjecture of Nicolaescu and the author, cf. \cite{NN1,Line}.
This provides a conceptual understanding how the Seiberg--Witten invariants appear in a natural way
in the world of  singularities, and why they can
 serve as topological candidates for the equivariant geometric genera.

Theorem B follows from Theorem A and (\ref{E}). It also
has the following interpretation. It is known that
the Seiberg--Witten invariant appears as the normalized Euler characteristic of the Heegaard--Floer
theory of Ozsv\'ath and Szab\'o, see \cite{OSzAB,NOSZ,Rus}
(or, of the Seiberg--Witten Floer homology).
Theorem B says that the normalized Euler characteristics of these cohomology theories and of the
lattice cohomology  coincide.
This supports the conjecture from \cite{NLC} which expects precise correspondence between
the corresponding cohomology modules and the normalization terms.

\section{Notations and preliminary results}
\labelpar{sec:main-results}

\subsection{Surface singularities and their graphs}\labelpar{ss:11}

Let \((X,o)\) be a complex normal surface singularity whose {\it link $M$
is a rational homology sphere}.  Let \(\pi:\widetilde{X}\to X\) be
a good resolution with dual graph  \(\Gamma\) whose vertices are
denoted by $\cV$. Hence  \(\Gamma\) is a tree and all the irreducible
exceptional divisors have genus \(0\). We will write  $s$, or $|\cV|$,
for the number of vertices.

Set \(L \coloneqq H_2 ( \widetilde{X},\setZ )\). It is freely
generated by the classes of the irreducible exceptional curves
\(\{E_v\}_{v\in\cV}\). They will also be identified with the
integral cycles supported on $E=\pi^{-1}(o)$. We set
$I_{vw}=(E_v,E_w)$. The intersection matrix $I=\{I_{vw}\}$
is negative definite, and any connected plumbing graph
with negative definite intersection form appears in this way for some
singularity. We write $e_v$ for $E_v^2$.

If  \(L'\) denotes
\(H^2( \widetilde{X}, \setZ )\), then the intersection form
 provides an embedding \(L \hookrightarrow
L'\) with factor $H^2(\partial
\widetilde{X},\setZ)\simeq H_1(M,\setZ)$; $[l']$ denotes the class of $l'$.
The form $(\,,\,)$ extends to
$L'$ (since $L'\subset L\otimes \setQ$).
$L'$ is freely generated by the duals \(E_v^*\), where we
prefer the convention $ ( E_v^*, E_w) =  -1 $ for $v = w$, and
$=0$ otherwise.

The {\it canonical class} $K\in L'$ is defined by the
{\it adjunction formulae}
\begin{equation}\label{eq:adjun}
(K+E_v,E_v)+2=0 \ \ \ \mbox{ for all $v\in\cV$.}
\end{equation}

For $l_1,l_2\in L'$ one writes $l_1\geq l_2$ if
$l_1-l_2=\sum r_vE_v$ with all $r_v\in\setQ_{\geq 0}$. Denote by $\cS'$  the Lipman
 cone $\{l'\in L'\,:\, (l',E_v)\leq 0 \ \mbox{for all
$v$}\}$.  It is generated over $\setZ_{\geq 0}$ by the
elements $E_v^*$. Since  {\em all the entries } of $E_v^*$
are {\em strict} positive, for any fixed $a\in L'$ one has:
\begin{equation}\label{eq:finite}
\{l'\in \cS'\,:\, l'\ngeq a\} \ \ \mbox{is finite}.
\end{equation}

\subsection{Motivation of Theorem A: Hilbert series.}\labelpar{FM}

One of the strongest analytic invariants of $(X,o)$ is its {\it
equivariant divisorial Hilbert series} $\cH(\bt)$. This is defined as follows
(for more details, see e.g. \cite[\S 2 and \S 3]{coho3} and \cite{CDGb}).

Fix a resolution $\pi$ of $(X,o)$ as in (\ref{ss:11}), let $c:(Y,o)\to (X,o)$ be the
universal abelian cover of $(X,o)$,  $\pi_Y :\widetilde{Y}\to Y$ the normalized
pullback of $\pi$ by $c$, and $\widetilde{c}:\widetilde{Y}\to \widetilde{X}$ the morphism which covers $c$.
Then $\cO_{Y,o}$ inherits  the {\em divisorial multi-filtration} (cf. \cite[(4.1.1)]{CDGb}):
\begin{equation*}\label{eq:03}
\cF(l'):=\{ f\in \cO_{Y,o}\,|\, {\rm div}(f\circ\pi_Y)\geq \widetilde{c}^*(l')\}.
\end{equation*}
Let $\hh(l') $ be the dimension of the $[l']$-eigenspace
of $\cO_{Y,o}/\cF(l')$. Then  the {\em equivariant
divisorial Hilbert series} is
\begin{equation*}\label{eq:04}
\cH(\bt)=\sum _{l'=\sum l_vE_v\in L'}
\hh(l')t_1^{l_1}\cdots t_s^{l_s}=\sum_{l'\in L'}\hh(l')\bt^{l'}\in
\setZ[[L']].
\end{equation*}
In $\cH(\bt)$ the exponents $l'$ of the terms $\bt^{l'}$  reflect the $L'/L\simeq H_1(M,\setZ)$ eigenspace
decomposition too.  E.g., $\sum_{l\in L}\hh(l)\bt^{l}$ corresponds
to the $H_1(M,\Z)$--invariants, hence it is the {\em Hilbert series}  of
$\cO_{X,o}$ associated with the $\pi^{-1}(o)$-divisorial
multi-filtration  (considered and intensively studied; see  e.g.
\cite{CHR} and the citations therein, or \cite{CDG}).

If $l'$ is in the `special zone' $l'\in -K+\cS'$, then by a vanishing (of a first cohomology), and by
Riemann-Roch, one obtains (see \cite{coho3}) that the expression
 \begin{equation}\label{eq:KV}
 \hh(l')+\frac{(K+2l')^2+|\cV|}{8}
\end{equation}
{\it depends only on the class $[l']\in L'/L$ of $l'$}. In several efforts to connect $\cH(\bt)$ with the
topology of the link (i.e. with the combinatorics of the graph $\Gamma$), the key bridge is done by the
series (cf.  \cite{CDG,CDGEq,CDGb,coho3}):
\begin{equation*}\label{eq:06}
\cP(\bt)=-\cH(\bt) \cdot \prod_v(1-t_v^{-1})\in \setZ[[L']].
\end{equation*}
Moreover, this identity (though it suggests that $\cP$ contains less information than $\cH$) can be
`inverted' (cf. \cite[(3.2.6)]{coho3}):
\begin{equation*} \label{eq:inv}
\hh(l')=\sum_{l\in L,\, l\not\geq 0} \bar{p}_{l'+l}, \ \
\mbox{where} \ \ \cP(\bt)=\sum_{l'}\bar{p}_{l'}\bt^{l'}.
\end{equation*}
($\cP$ is supported on
$\cS'$, see e.g.  \cite[(3.2.2)]{coho3}, hence the sum is finite, cf. (\ref{eq:finite})).
In particular, by (\ref{eq:KV}),
\begin{equation}\label{eq:KV2}
\sum_{l\in L,\, l\not\geq 0} \bar{p}_{l'+l}=-\mathrm{const}_{[-l']} -
\frac{(K+2l')^2+|\cV|}{8}
\end{equation}
for any $l'\in-K+\cS'$, where $\mathrm{const}_{[-l']}$ depends only on the class $[-l']$ of $-l'$.
The right hand side can be interpreted as a `multivariable Hilbert polynomial' of degree 2 associated with
the series $\cH(\bt)$, or with $\cP(\bt)$.

The point is that $\cP(\bt)$ has a {\it topological candidate},
namely  $Z(\bt)$ (for its definition see (\ref{eq:INTR}) from the Introduction),
 which for several singularities agrees with $\cP(\bt)$, cf. \cite{CDGEq,CDGb,coho3}.
In this way, for such singularities, one gets a topological characterization
of the constant terms from (\ref{eq:KV2}). Since these constants (equivariant geometric
genera, cf. \cite{coho3}) by the
conjectures of \cite{NN1,Line,trieste} equal  the normalized Seiberg--Witten invariants of the link
(for `nice' analytic structures),
one expects that the series $Z(\bt)$ admits a  multivariable Hilbert
polynomial too,  similar to the right hand side of
(\ref{eq:KV2}) with constant terms  the normalized Seiberg--Witten invariants.
This fact was announced in \cite{coho3}, and its proof  is the subject of the present article.

\subsection{The lattice cohomology}\label{ss:LatCoh}
First we recall the definition of the {\it lattice cohomology} from \cite{NLC} and \cite{NES}.
Let $Char:=\{k\in L'\,:\,
(k+l,l)\in 2\Z \ \ \mbox{for all $l\in L$}\}$ denote the set of characteristic elements of $L$.
It is  an $L'$--torsor: $Char=K+2L'$.

The set of $q$--cubes, $\calQ_{q}$, consists of
pairs $(k,I)\in Char\times \calP(\calj)$, $|I|=q$, (here $\calP(\calj)$ denotes the power set of $\calj$).
$\square_q=(k,I)$ can be identified with the `vertices' $\{k+2\sum_{j\in I'}E_j)_{I'}$, where
$I'$ runs over all subsets of $I$, of a $q$--cube in $L'\otimes \R$.
One defines the weight
function  induced by the intersection form
\begin{equation}\label{eq:q}
\q: Char\to \Q, \ \ \ \ \ \ \q(k):=-(k^2+|\cV|)/8,
\end{equation}
which extends to a weight--function  of the  $q$--cubes
\begin{equation*}
w(\square_q)=w((k,I))=\max_{I'\subset I}\big\{\, \q(k+2\sum_{j\in I'}E_j)\,\big\}.
\end{equation*}
Let $\calF_q$ be the direct product
of $\Z_{\geq 0}\times \calQ_q$ copies of $\Z$. We write the pair $(m,\square)$ as $U^m \square$.
 $\calF_q$ becomes a $\Z[U]$--module by $U(U^m \square)=U^{m+1} \square$. One defines
 $\partial:\calF_q\to \calF_{q-1}$ as follows.
 For $\square=(k,I)=(k,\{v_1,\ldots,v_q\})$ one set
\begin{equation*}\label{eq:partial}
\partial (k,I)=\sum_{l=1}^q (-1)^l\big( \,
U^{w(k,I)-w(k,I\setminus v_l)} (k,I\setminus v_l)-U^{w(k,I)-w(k+2E_{v_l},I\setminus v_l)}
(k+2E_{v_l},I\setminus v_l)\, \big).
\end{equation*}
Then $\partial\circ \partial =0$, hence $(\calF_*,\partial)$ is a chain complex of $\Z[U]$--modules.
The dual cochain complex is defined by
$\calF^q=\Hom_{\Z[U]}(\calF_q,\et^+_0)$, consisting of finitely supported morphisms with
$\phi(U^m\square)=U^m\phi(\square)$.
Here,  $\et_0^+$ denotes the $\Z[U]$--module
$\Z[U,U^{-1}]/U\Z[U]$ with grading $\deg(U^{-d})=2d$ ($d\geq 0$), as usual.
More generally, for any  $r\in\Q$ one defines $\et^+_r$, the same module as $\et_0^+$, but graded
(by $\Q$) in such a way that the $d+r$--homogeneous elements of $\et^+_r$ are isomorphic with the
$d$--homogeneous elements of $\et_0^+$.

$\calF^q$ is  a $\Z[U]$--module with a $\Q$--grading: $\phi\in \calF^q$ is
homogeneous of degree $r$ if for each $\square_q\in\calQ_q$
with $\phi(\square_q)\not=0$, $\phi(\square_q)$ is a homogeneous
element of $\et^+_0$ of degree $r-2\cdot w(\square_q)$.
The coboundary operator
$\delta:\calF^q\to \calF^{q+1}$ is defined by
 $\delta(\phi)(\square)=\phi(\partial (\square))$.
The cohomology of $(\calF^*,\delta)$ is the lattice cohomology of $\Gamma$, and it is
denoted by $\bH^*(\Gamma)$.
Since the vertices of a cube belong to the same class $Char/2L=K+2L'/2L$ (where a class has the form
$[k]=\{k+2l\}_{l\in L}\subset Char$), the complex $(\calF^*,\delta)$ and the cohomology
$\bH^*(\Gamma)$ have  natural direct sum decompositions of $\Z[U]$--modules:
$$(\calF^*,\delta)=\bigoplus_{[k]\in Char/2L}\ (\calF^*[k],\delta [k]) \ \ \mbox{and} \ \
\bH^*(\Gamma)=\bigoplus_{[k]\in Char/2L}\ \bH^*(\Gamma,[k]).$$
In fact, if $[k_1]=[k_2]$ then $w(k_1)-w(k_2)\in \Z$, and the set of degrees of $\calF^*[k]$ is
$2\Z$,  shifted by a rational number.
Since $\Gamma $ is negative definite, for each class $[k]\in Char/2L$ one has a well--defined
rational number
\begin{equation*}
d[k]:=-\max _{k\in[k]} \frac{k^2+|\cV|}{4}=2\cdot \min_{k\in[k]} \q(k).\end{equation*}
One defines an augmentation $\epsilon:\et^+_{d[k]}\to \calF^0[k]$ of the complex $(\calF^*[k],\delta[k])$;
the cohomology of the augmented complex is called the {\it reduced cohomology } $\bH^*_{red}(\Gamma,[k])$.
One has $\bH^q_{red}(\Gamma,[k]):=\bH^q(\Gamma,[k])$ for $q>0$, and  a direct sum decomposition
of $\Z[U]$--modules:
\begin{equation*}\label{eq:et}
\bH^0(\Gamma,[k])=\et^+_{d[k]}\oplus \bH^0_{red}(\Gamma,[k]).
\end{equation*}
$\bH^*_{red}(\Gamma,[k]):=
\oplus_{q\geq 0}\bH^q_{red}(\Gamma,[k])$  has finite $\Z$--rank.
The `normalized' Euler--characteristic of $\bH^*(\Gamma,[k])$ is
\begin{equation}\label{eq:EU}
eu(\bH^*(\Gamma,[k])):= -d[k]/2+\sum_q(-1)^q\rank_\Z (\bH^q_{red}(\Gamma,[k])).
\end{equation}
$\bH^*(\Gamma)$ and
$\bH^*_{red}(\Gamma)$ depend only on $M=M(\Gamma)$, and not on the plumbing graph $\Gamma$.
The involution $l'\mapsto -l'$ induces an isomorphism $\bH^*(\Gamma,[k])=\bH^*(\Gamma,[-k])$, hence
\begin{equation}\label{eq:INVOL}
eu(\bH^*(\Gamma,[k]))=eu(\bH^*(\Gamma,[-k])).
\end{equation}
\bekezd\label{be:cubes} {\bf Rectangles.}
In  combinatorial enumerations of the weighted $q$--cubes it is convenient to replace the set of
all cubes by only those ones which are supported on a fixed compact subset of $L'\otimes \R$.
In the simplest case we take rectangles: for any fixed class $[k]=k_0+2L\subset  Char$,
one takes two characteristic elements $k_1, \ k_2\in [k]$ with
$k_1\geq k_2$.
We denote by $R=R(k_1,k_2)$ the rectangle $\{k\in [k]\,:\, k_1\geq k\geq k_2\}$.
Similarly, for one fixed element $k_1\in[k]$, one can take $R=R(k_1)=\{k\in[k]\,:\, k_1\geq k\}$.
Once such an $R$ is identified, one  considers
the complex $(\calF^*(R),\delta(R))$, constructed similarly as $(\calF^*,\delta)$, consisting of all the
cubes $(k,I)$ with all vertices in $R$ (this fact will be denoted by  $(k,I)\in R$).
Using  $\min(\q|R):= \min_{k\in R}\q(k)$ one also defines the corresponding
augmented complex,  and one gets the corresponding lattice cohomologies $\bH^*(R,[k])$ and
$\bH^*_{red}(R,[k])$ with $\bH^0(R,[k])=\et^+_{2\min(\q|R)}\oplus \bH^0_{red}(R,[k])$. For more details, see
\cite{NLC}.  
We also define the `normalized Euler characteristic' of this lattice cohomology, namely
\begin{equation*}
eu(\bH^*(R,[k]):=-\min(\q|R)+\sum_{q\geq 0}(-1)^q\rank_\Z\, \bH^q_{red}(R,[k]).
\end{equation*}
Let $\cS'_{st}$ be the {\it strict Lipman cone} $\cS'_{st}=\{l'\in L'\,:\, (l',E_v)<0 \
\mbox{for all $v$}\}.$

\begin{proposition}\label{prop:tech} Fix a class $[k]$. Assume that $k_1\in[k]$ satisfies
$k_1\in -K+\cS'_{st}$, that is $(k_1,E_v)\leq e_v+1$ for any $v\in\cV$.
Then the following facts hold:

(a) \ For any $k\in [k]$, $k>k_1$, there exists some $E_v$ in the support $|k-k_1|$
of $k-k_1$, so that
$\q(k-2E_v)\leq \q(k)$.

(b) \ There exists  an increasing (computation)
sequence $\{z_n\}_{n\geq 0}$, $z_n\in L$, with $z_0=0$,
and  $z_{n+1}=z_n+E_{v(n)}$ for some $v(n)\in\cV$ when $n\geq 0$, satisfying:

\vspace{1mm}

(i) \ The coefficients of $z_n$ tends to infinity, that is $-\lim_{n\to \infty}
(z_n,E_v^*)=\infty$ for any $v$.

(ii) \ For fixed $n\geq 0$, let $x\in L$ be such that $x\leq z_n$ and $(x,E_{v(n)}^*)=(z_n,E_{v(n)}^*)$.
Then $\q(k_1+2x)\leq \q(k_1+2x+2E_{v(n)})$.

(iii) \ The restriction $\bH^*(R(k_1+2z_{n+1}),[k])\to \bH^*(R(k_1+2z_{n}),[k])$ is an isomorphism
of weighted $\Z[U]$--modules compatible with the augmentation.

\vspace{1mm}

In particular,  $\bH^*(\Gamma,[k])=\bH^*(R(k_1),[k])$ compatibly with the augmentation.

Moreover, in a similar way, one can find $k_2$ (with all its $E_v$--coefficients sufficiently
small) such that  $\bH^*(\Gamma,[k])=\bH^*(R(k_1,k_2),[k])$.
\end{proposition}
\begin{proof} (a) Assume that $\q(k-2E_v)>\q(k)$ for any $E_v\in|k-k_1|$. This says that $(k-2E_v)^2<k^2$,
that is $(k,E_v)>E_v^2$, or $(k+K,E_v)\geq -1$. Since $(k_1+K,E_v)\leq -1$, one gets $(k-k_1,E_v)\geq 0$ for any $E_v$, hence $(k-k_1)^2\geq 0$, which is a contradiction.

(b) Fix a $D=\sum_vd_vE_v\in \cS'\cap L$ with $d_v\in\Z_{>0}$ for all $v$.   
By (a) we get an increasing
 computation sequence $\{y_n\}_{n=0}^{n_0}$ connecting 0 and $D$ so that
 $\q(k_1+2y_n)\leq \q(k_1+2y_{n+1})$.
 Then the sequence $\{z_n\}_{n\geq 0}:=
 \{mD+y_n\}_{m\geq 0;\, 0\leq n\leq n_0}$ satisfies (i) and
 $\q(k_1+2z_n)\leq \q(k_1+2z_{n+1})$. All other properties follow similarly as in \cite[p. 518]{NLC}.
\end{proof}

\bekezd\label{ss:23}
{\bf Counting weighted cubes and the
Euler--characteristic of the lattice cohomology.} \
The next result generalizes the classical fact that the alternating sum of the
number of $q$--cubes is the Euler characteristic of the cohomology.
For any finite set $A\subset [k]$ define
$$\cE(A):=\sum_{(k,I)\in A} (-1)^{|I|+1}\q((k,I)) \ \ \mbox{and } \ \
\calm_A(t):=\sum_{(k,I)\in A}(-1)^{|I|}t^{\q((k,I))}.$$

\begin{theorem}\label{th:LC} Let $R$ be a finite rectangle $R(k_1,k_2)$. Then
$\cE(R)=eu(\bH^*(R,[k]))$.
\end{theorem}
\begin{proof}
We will reduce the result to the classical case via a certain geometric interpretation of the lattice cohomology from \cite[(3.1.12)]{NLC}. Note that $R$ can also be interpreted as a real rectangle
 in $L\otimes_\Z \R$ limited by the vertices $k_1$ and $k_2$. It has a natural cubic decomposition into the `real cubes' $(k,I)$ with all vertices in the rectangle.  For any non--negative integer  $n$ define $S_n$ as the union of all these real cubes of $L\otimes \R$ with vertices from $R$, and with weights $\leq
 n+\min(\q|R)$. Let $\chi(S_n)$ be its (classical) Euler--characteristic 
  and $\chi_{red}(S_n)=\chi(S_n)-1$ the Euler-characteristic of its reduced simplicial cohomology.
 Then, by \cite[(3.1.12)]{NLC}, one has for any $q\geq 0$ the $\Z$--module isomorphisms:
 \begin{equation}\label{eq:SN}
 \bH^q(R,[k])=\bigoplus_{n\geq 0}H^q(S_n,\Z), \ \ \
\bH^q_{red}(R,[k])=\bigoplus_{n\geq 0}\widetilde{H}^q(S_n,\Z).\end{equation}
In particular, if we write $\calm_R(t)/(1-t)$ as $\sum_{n\geq 0}a_nt^{n+\min(\q|R)}$, then
$$a_n=\sum_{(k,I)\in R \atop \q((k,I))\leq n+\min(\q|R)} \ (-1)^{|I|}=\chi(S_n). $$
Therefore, $$\frac{\calm_R(t)-t^{\min(\q|R)}}{1-t}=\sum_{n\geq 0}\chi_{red}(S_n)t^{n+\min(\q|R)},$$
hence, by (\ref{eq:SN}):
$$\lim_{t\to 1}\frac{\calm_R(t)-t^{\min(\q|R)}}{1-t}=\sum_{n\geq 0}\chi_{red}(S_n)
=\sum_{q\geq 0}(-1)^q\rank_\Z\, \bH^q_{red}(R,[k]),$$
that is, $-\frac{d}{dt}\calm_R(1)=eu(\bH^*(R,[k]))$, the wished statement of the theorem.
\end{proof}

\bekezd\label{ss:21} {\bf Counting weighted cubes and the zeta function $Z(\bt)$.} \
The next result provides the key step for the identification of the `lattice cohomology package'
and numerical invariants provided by the series  $Z(t)$.
Before we state it, let us recall  that $Char=K+2L'\subset L'$, hence $(k-K)/2$ runs over
$L'$ when $k$ runs over $Char$.

\begin{theorem}\label{th:ZW}
Let $\Gamma$ be a connected negative definite graph. Then
\begin{equation}\label{eq:ZW}
Z(\bt)=\sum_{k\in Char}\ \sum_{I\in  \calP(\calj)}\ (-1)^{|I|+1}\, \q((k,I))\,\bt^{\frac{1}{2}
(k-K)}.
\end{equation}
\end{theorem}
\noindent Since $\sum_I(-1)^{|I|}=0$, here $\q((k,I))$ can be replaced
by $\q((k,I))+c$ for any constant $c$.
\begin{proof}
For each $k=\sum_va_vE^*_v\in Char$,
where $a_v\equiv e_v $ (mod 2), write
$\bz^a:=\prod_vz_v^{a_v}$ and set the counting function
$$\calU_\Gamma(\bz):= \sum_{k\in Char}\ \sum_{I\in  \calP(\calj)}\ (-1)^{|I|+1}\, \q((k,I))\, \bz^a.$$
We determine $\calU_\Gamma$ by induction on $|\calj|$.
If $|\calj|=1$ and the decoration of the unique vertex is $e<0$,  then $|e|E^*=E$, $(E^*)^2=1/e$, and $k=aE^*\in Char$ with
$a\equiv e$ (mod 2). Hence
\begin{equation}\label{eq:egy}
8\,\calU_\Gamma(z)=\sum_{a\equiv e \ (mod\, 2)} \Big[
-\frac{a^2}{|e|}+\max\Big\{ \, \frac{a^2}{|e|}, \frac{(a-2e)^2}{|e|}\, \Big\}
\ \Big]\cdot z^a.\end{equation}
If $a\leq e$ then $a^2\geq (a-2e)^2$, hence the coefficient is vanishing. Otherwise we write
$a=e+2n$ for $n\in\Z_{>0}$ and we get that $\calU(z)=\sum_{n\geq 1}nz^{e+2n}=z^{e+2}/(1-z^2)^2$.

Next, we assume $|\calj|\geq 2$.
Let $u$ be a fixed end--vertex of $\Gamma$
(that is, $\delta_u=1$). Set $\Gamma_0:=\Gamma\setminus u$ the graph obtained by deleting $u$ and its
supporting edge. If $k=\sum_va_vE^*_v\in Char(\Gamma)$,
we write $k_0=\sum_{v\not=u}a_vE^*_v\in Char(\Gamma_0)$. The series $\calU_\Gamma$ can be written
as a sum $\calU_\Gamma^{(1)}+\calU_\Gamma^{(2)}$, where the first
series is the sum over those subsets $I$ which does not contain $u$,
while the second is the sum over the other terms.
 For the first case when $I\not\ni u$
\begin{equation}\label{eq:w1}
\calU_\Gamma^{(1)}(\bz)=
\calU_{\Gamma_0}(\bz_0)\cdot \sum_{a_u\equiv e_u\, (2)}  z_u^{a_u},
\end{equation}
where $\bz_0$ are the variables $\{z_v\}_{v\not=u}$ corresponding to $\Gamma_0$.
Indeed, one has  $\q((k,I))-\q(k)=\q((k_0,I))-\q(k_0)$ and $\sum_{I\not\ni u}(-1)^{|I|}\q(k)=
\sum_{I\not\ni u}(-1)^{|I|}\q(k_0)=0$, hence
\begin{equation}\label{eq:3}
\begin{split}
& \sum_{k\in Char(\Gamma)}\   \sum_{I\not\ni u }\ (-1)^{|I|+1}\, \q((k,I)) \bz^a=\\
& \sum_{k\in Char(\Gamma)}\ \sum_{I\not\ni u }\ (-1)^{|I|+1}\,
\big(\q((k,I))-\q(k)\big) \bz^a
= \\
& \sum_{k\in Char(\Gamma)}\ \sum_{I\not\ni u }\ (-1)^{|I|+1}\,
\big(\q((k_0,I))-\q(k_0)\big) \bz^a = \\
&  \sum_{k_0\in Char(\Gamma_0)}\ \sum_{a_u\equiv e_u\, (2)}  z_u^{a_u} \sum_{I\not\ni u }\
(-1)^{|I|+1}\,
\q((k_0,I))\,\prod_{v\not=u}z_v^{a_v}.
\end{split}
\end{equation}
In the second sum $u\in I$; set $I=I'\cup u$ with $u\not\in I'$. Since
$$(k+2E_{I'}+2E_u)^2-(k+2E_{I'})^2=4\big(-a_u+e_u+2(E_{I'},E_u)\,\big), $$
where  $(E_{I'},E_u)\in\{0,1\}$,  one gets that
$$\left\{\begin{array}{ll}
\mbox{if $-a_u+e_u<0$, then} & \q((k,I))=\q((k+2E_u,I')),\\
\mbox{if $-a_u+e_u\geq 0$, then} & \q((k,I))=\q((k,I')). \end{array}\right.$$
Hence $\calU_\Gamma^{(2)}(\bz) $ splits into two sums:
\begin{equation*}\begin{split}
\sum_{k_0\in Char(\Gamma_0)}\ \sum_{a_u\equiv e_u\, (2) \atop a_u\leq e_u} &
\sum_{I=I'\cup u }\ (-1)^{|I|+1}\,
\q((k,I'))\,\bz^a+
\\ \ &
\sum_{k_0\in Char(\Gamma_0)}\ \sum_{a_u\equiv e_u\, (2) \atop a_u> e_u}
\sum_{I=I'\cup u }\ (-1)^{|I|+1}\,
\q((k+2E_u,I'))\,\bz^a.\end{split}
\end{equation*}
For the second one we use $E_u=-e_uE_u^*-E^*_{u_0}$, where $u_0$ is the adjacent vertex of $u$ in $\Gamma$.
Then, computing both sums
by similar argument  as in (\ref{eq:3}), we get
\begin{equation}\label{eq:w2}
\calU_\Gamma^{(2)}(\bz)=-\calU_{\Gamma_0}(\bz_0)\cdot \sum_{a_u\equiv e_u\, (2)\atop
a_u\leq e_u}  z_u^{a_u}-\calU_{\Gamma_0}(\bz_0)\cdot z_{u_0}^2\cdot \sum_{a_u\equiv e_u\, (2)\atop
a_u> e_u}  z_u^{a_u}.
\end{equation}
The contributions (\ref{eq:w1}) and  (\ref{eq:w2})  combined provide
\begin{equation*}\label{eq:ind}
\calU_\Gamma(\bz)=\calU_{\Gamma_0}(\bz_0)(1- z_{u_0}^2)\cdot \sum_{a_u\equiv e_u\, (2)\atop
a_u> e_u}  z_u^{a_u}=\calU_{\Gamma_0}(\bz_0)\cdot \frac{1- z_{u_0}^2}
{1-z_u^2}\cdot z_u^{e_u+2}.
\end{equation*}
This as an inductive step, together with the identity valid for $|\calj|=1$, give
\begin{equation*}\label{eq:U}
\calU_\Gamma(\bz)=\prod_{v\in\calj}z_v^{e_v+2}\cdot \prod_{v\in\calj}(1-z_v^2)^{\delta_v-2}.
\end{equation*}
From this and the from the definition of $\calU_\Gamma$ we obtain
$$\sum_{k\in Char}\ \sum_{I\in\calP(\calj)}\
(-1)^{|I|+1}\, \q((k,I))\, \prod_{v\in\cV} x_v^{\frac{a_v-e_v-2}{2}}=
\prod_{v\in\calj}(1-x_v)^{\delta_v-2}.$$
Then (\ref{eq:ZW}) follows  via the substitution $x_v=\bt^{E_v^*}$, since $K=\sum_v(2+e_v)E^*_v$.
\end{proof}

\section{The proofs of Theorems A and B}\labelpar{s:frs}

\subsection{The definition of the invariant $\frs$ and its relation with $\bH^*(\Gamma)$.}\labelpar{ss:frs}
Let $\Gamma$ be a  graph  as in (\ref{ss:11}) and  $Z(\bt)$ the series defined in the introduction.
\begin{theorem}\labelpar{th:1} 
(a) For any $l'\in L'$, the expression
\begin{equation}\label{eq:s}
-\sum_{l\in L,\, l\ngeq 0} p_{l'+l}-\frac{(K+2l')^2+|\cV|}{8}
\end{equation}
depends only on the class $[l']\in L'/L$ of\,  $l'=\sum_{v\in\cV}a_vE^*_v$, provided that
 $a_v\geq -e_v-1$ for all $v\in\cV$. 
(Since  $Z$ is supported on
$\cS'$, the  sum in (\ref{eq:s}) is finite by (\ref{eq:finite}).) 

(b)  Consider the  map $\frs:L'/L\to\setQ$, \ $[l']\mapsto \frs_{[l']}$,  where
$\frs_{[-l']}$ is the expression (\ref{eq:s}).
Then the set $\{\frs_{[l']}\}_{[l']}$  is independent of the negative definite plumbing representation
$\Gamma$, it
 depends only on the oriented plumbed 3--manifold $M=M(\Gamma)$.
In fact, for any $l'$, one has
\begin{equation}\label{eq:SEU}
\frs_{[-l']}=-eu(\bH^*(\Gamma,[K+2l'])).
\end{equation}
\end{theorem}
\begin{proof} We fix $k_1=K+2l'$ and $k_2$ as in (\ref{prop:tech}), that is  $(k_1,E_v)\le e_v+1$
and all the coefficients of $k_2 $ are sufficiently small ($k_2$ does not play on essential role,
it only assures the finitness  of the rectangles), and
a computation sequence $\{z_n\}_{n\geq 0}$ as in (\ref{prop:tech})(b).  Set
$$R':=\{k\in [k]\,:\, k\geq k_2, \ k-k_1=\sum  l_vE_v \ \mbox{so that}\ \exists \, l_v\leq 0\}.$$
Although $R'$ is not finite, $R'\cap \cS'$ is finite by (\ref{eq:finite}). Fix some
$\tilde{n}$ so that $R'\cap \cS'\subset R(k_1+2z_{\tilde{n}},k_2)$ and define
$\widetilde{R}:= R'\cap R(k_1+2z_{\tilde{n}},k_2)$. Take also
$$\partial\widetilde{R}:=
\{k\in [k]\,:\, k_1\leq k\leq k_1+2z_{\tilde{n}}, \ k-k_1=\sum  l_vE_v \ \mbox{so that}\ \exists \,
l_v= 0\}.$$
Then, by Theorem~\ref{th:ZW}
\begin{equation}\label{eq:tilde}
\sum_{l\in L,\, l\ngeq 0} p_{l'+l}=\cE(\widetilde{R})-\cE(\partial\widetilde{R}).\end{equation}
Now, we claim that by combinatorial cancellation in the sum,
$\cE(\widetilde{R})=\cE(R)$, where $R=R(k_1,k_2)$.
This follows by induction using the sequence $\{z_n\}_{0\leq n\leq \tilde{n}}$,
since $\cE(R'\cap R(k_1+2z_{n+1},k_2))=\cE(R'\cap R(k_1+2z_n,k_2))$.
Indeed, for any $I$ containing $v(n)$ and cube $(k,I)\in
R'\cap (R(k_1+2z_{n+1},k_2)\setminus  R(k_1+2z_n,k_2))$ one has $\q((k,I))=\q((k+2E_{v(n)},I\setminus v(n)))$
by (\ref{prop:tech})(b)(ii).
Similarly, one gets $\cE(\partial\widetilde{R})=-\q(k_1)$. Hence (\ref{eq:tilde}) reads as
\begin{equation}\label{eq:tilde2}
-\q(K+2l')+\sum_{l\in L,\, l\ngeq 0} p_{l'+l}=\cE(R).\end{equation}
The right hand side is $eu(\bH^*(R,[k_1])$ by Theorem~\ref{th:LC}, which equals
 $eu(\bH^*(\Gamma,[k_1])$ by  Proposition~\ref{prop:tech}. In particular, the left hand side too
depends only on $[l']$.
 It is invariant under blow up of the graph since the lattice cohomology is so, cf. \cite{NLC,NES}.
\end{proof}
\subsection{The surgery formula for $\frs$}\labelpar{s:sur}

\bekezd\labelpar{ss:first}{\bf The additivity formula.} \
Similarly as in the proof of Theorem~\ref{th:ZW}, we change the
variables of the series $Z(\bt)$.  By setting $x_v:=\bt^{E^*_v}$ for all $v\in\cV$, $Z(\bt)$
transforms into $\cZ_\Gamma(\bx)=\prod _{v\in\cV} (1-x_v)^{\delta_v-2}$,
whose Taylor series  at the origin is denoted by $\sum q_a\bx^a$, where
$\bx^a=x_1^{a_1}\cdots x_s^{a_s}$. The exponents $a_v$ are the coordinates
of $L'$ in the basis $\{E^*_v\}_v$, i.e. if $l'=\sum_va_vE^*_v$ then $a_v=-(l',E_v)$ and
$\bt^{l'}$ transforms into $\bx^a$.
In particular, $q_a=p_{l'}$, and we also use the notation $a$ for $l'\in L'$.
 For  any fixed \ $l'=\sum_v a_vE^*_v$ we define
\begin{equation}\label{eq:haha}
h^\Gamma_a:=\sum_{l\in L,\, l\not\geq 0}\, p_{l'+l}=\sum_{b\in S_\Gamma(a)}\, q^\Gamma_b,\end{equation}
where $S_\Gamma(a)=\{b\in L'(\Gamma)\,:\, b=a+\sum n_vE_v,
\ n_v\in\setZ, \exists\ n_v<0\}$. We assume that $s\geq 2$
and we fix an {\it end--vertex} $u$ of $\Gamma$. We set
$$h^u_a:=\sum \, q^\Gamma_b \ \ \ \ \ \
\ (\mbox{sum over \ \   $b=a+\sum n_vE_v$, \ $n_v\in\Z$, \ $n_u<0$}). $$
The inclusion of the subgraph $\Gamma\setminus u$ induces $i:L(\Gamma\setminus u)\to L(\Gamma)$,
$i(E_v)\mapsto E_v$
(the symbol $i$ sometimes is omitted), and its dual, the
restriction $R:L'(\Gamma)\to L'(\Gamma\setminus u)$ with 
\begin{equation}\label{eq:RRR}
\left\{\begin{array}{l}
\mbox{$R(E^{*,\Gamma}_v)=E^{*,\Gamma\setminus u}_v$ for $v\not=u$, and $R(E^{*,\Gamma}_u)=0$};\\
\mbox{$R(E_v^\Gamma)=E_v^{\Gamma\setminus u}$ for $v\not=u$, and $R(E^\Gamma_u)=-E^{*,\Gamma\setminus u}_w$,}
\end{array}\right.\end{equation}
where $w$ is the adjacent vertex of $u$.  We abridge  $R(\sum_v a_vE^*_v)$ by $R(a)$.
\begin{proposition}\labelpar{prop:ad}
Assume that $w$, the adjacent vertex of $u$ in $\Gamma$, has valency two.
Then,   if $a_u\gg 0$ (compared with $R(a)$), one has
$$h^\Gamma_a=h^u_a+h^{\Gamma\setminus u}_{R(a)}.$$
\end{proposition}
\begin{proof}Since $\delta_w=2$, $\cZ_\Gamma$ and $\cZ_{\Gamma\setminus u}$ have the form
\begin{equation}\label{eq:ad0}
\cZ_\Gamma=\widetilde{\cZ}\cdot \frac{1}{1-x_u}; \ \ \
\cZ_{\Gamma\setminus u}=\widetilde{\cZ}\cdot \frac{1}{1-x_w},
\end{equation}
where $\widetilde{\cZ}$ does not depend on the variables $x_w$ and $x_u$.
Therefore,  for any relevant  $b\in S_\Gamma(a)$ (i.e. when $q^\Gamma_b\not=0$) one has:
$b_w=0$ and $b_u\geq 0$. Hence:
$$h^\Gamma_a-h^u_a=\sum_{b\in S'_\Gamma(a)} q^\Gamma_b$$
where $S'_\Gamma(a)=\{b\in S_\Gamma(a):\, b_w=0,\ b_u\geq 0, \ n_u\geq 0, \exists\, n_v<0\}$.

By similar argument based on (\ref{eq:ad0}), for any
relevant  $c\in S_{\Gamma\setminus u}(R(a))$, one has $c_w\geq 0$,
hence it is enough to consider the subset $S'_{\Gamma\setminus u}(R(a))=\{
c\in S_{\Gamma\setminus u}(R(a)):\, c_w\geq 0\}$ in the computation of
$h^{\Gamma\setminus u}_{R(a)}$.
Applied $(\cdot, E_u)$ to the identity  $b=a+\sum_vn_vE_v$, one gets
\begin{equation}\label{eq:bu}b_u=a_u-n_w-I_{uu}n_u.
\end{equation}
Since $n_u=-(E^{*,\Gamma}_u,b-a)$ in terms of $b$, for any fixed $a$ one has a well--defined map
$$\Phi: S'_\Gamma(a)\to S'_{\Gamma\setminus u}(R(a)),  \ \
b=(\{b_v\}_{v\not=w,u},b_w=0,b_u)\mapsto (\{b_v\}_{v\not=w,u},n_u).$$
Here, $n_u$ maps to the $w$--entry in $ S'_{\Gamma\setminus u}(R(a))$. In other words,
 $\Phi(b)=R(b)+n_uE_w^{*,\Gamma\setminus u}$ (use (\ref{eq:RRR})). Moreover, again by (\ref{eq:RRR}),
 $\Phi(b)-R(a)=R(\sum_{v\not=u}n_vE^\Gamma_v)= \sum_{v\not=u}n_vE^{\Gamma\setminus u}_v$,
 hence the integers $\{n_v\}_{v\not=u}$ are the same at the level of $\Gamma$ and $\Gamma\setminus u$.
This fact, and (\ref{eq:bu}) implies the injectivity of $\Phi$. For the surjectivity,
for any $(\{b_v\}_{v\not=w,u},n_u)\in S'_{\Gamma\setminus u}(R(a))$
set $b_u$ defined by (\ref{eq:bu}),
then $(\{b_v\}_{v\not=w,u},0,b_u)$ satisfies automatically all the
conditions of $S'_\Gamma(a)$, except maybe one, namely
$b_u\geq 0$. In order to guarantee this one too, we argue as follows:
for $R(a)$ fixed, consider all the
elements  $(\{b_v\}_{v\not=w,u},n_u)\in S'_{\Gamma\setminus u}(R(a))$,
and associated with them
the maximum $M$ of all the possible values $n_w+I_{uu}n_u$. Then, if we
 take $a_u\geq M$, then by
(\ref{eq:bu}) the inequality $b_u\geq 0$ is also satisfied.

Finally, notice that from (\ref{eq:ad0}), for any $b\in S'_\Gamma(a)$  one has
$$q^\Gamma_b=q^{\Gamma\setminus u}_{\Phi(b)},$$
since both of them agree with the
$\prod_{v\not=w,u}x^{b_v}$--coefficient of $\widetilde{\cZ}$.
\end{proof}

\begin{corollary}\labelpar{cor:ha} Fix an end--vertex $u$ as above.
For $l'=\sum_va_vE^{*,\Gamma}_v$, with all $a_v$ large (and $a_u$ large compared with the others),
  one has:
$$h^u_a=-\frs^\Gamma_{[-l']}-\frac{K_\Gamma^2+|\cV|}{8}
+\frs^{\Gamma\setminus u}_{[-R(l')]}+\frac{K_{\Gamma\setminus u}^2+|\cV\setminus u|}{8}
-\frac{(l',E^*_u)\cdot (l'+K_\Gamma,E^*_u)}{2(E^*_u,E^*_u)}.
$$
\end{corollary}
\begin{proof} Theorem~(\ref{th:1}) applied for $l'$  and $R(l')$, and  Proposition~\ref{prop:ad}
provide
$$h^u_a=-\frs^\Gamma_{[-l']}-\frac{(K_\Gamma+2l')^2+|\cV|}{8}
+\frs^{\Gamma\setminus u}_{[-R(l')]}+\frac{(K_{\Gamma\setminus u}+2R(l'))^2+|\cV\setminus u|}{8}.
$$
Then use the  identity
$K_{\Gamma\setminus u}=R(K_\Gamma)$, and
$$l'' =iR(l'')+\frac{(l'',E^*_v)}{(E^*_u,E^*_u)}\cdot E^*_u$$
for both $l''=l'$ and $l''=K_\Gamma$, and $(i(l),E^*_u)=0$ for any $l\in L'(\Gamma\setminus u)$.
\end{proof}

\bekezd\labelpar{ss:pc}{\bf The series  $\cH_{[l'],u}(t)$ and its periodic constant.} \
For any series $S(\bt)\in \setZ[[L']]$, $S(\bt)=\sum_{l'}c_{l'}\bt^{l'}$, we have the
natural decomposition $$S=\sum_{h\in L'/L}S_h, \ \ \mbox{where} \ \  S_h:=\sum _{l'\,:\ [l']=h}\,
 c_{l'}\bt^{l'}.$$
In particular, for any fixed class $[l']\in L'/L$, one can consider the component
$Z_{[l']}(\bt)$ of $Z(\bt)$. In fact, see e.g. \cite[(3.1.20)]{CDGb},
\begin{equation}\label{eq:Po}
Z_{[l']}(\bt)=
\frac{1}{d}\sum _{\rho\in (L'/L)\,\widehat{}} \rho([l'])^{-1}\cdot
\prod_{v\in \cV} (1-\rho([E^*_v])\bt^{E^*_v})^{\delta_v-2},
\end{equation}
where $(L'/L)\,\widehat{}$ \ is the Pontjagin dual of $L'/L$, and $d=\det(-I)=|L'/L|$.
\begin{definition}\labelpar{def:h}
For any class $[l']\in L'/L$ and vertex $u\in\cV$ of $\Gamma$ set
$$\cH_{[l'],u}(t):=Z_{[l']}(\bt)\big|_{t_u=t^d \ \ \ \ \ \ \ \atop \ t_v=1\ \mbox{\tiny{for}} \ v\not=u
}\ \in \setZ[[t]].$$
\end{definition}

\begin{definition} {\bf Periodic constant \cite[3.9]{NO1}, \cite[4.8(1)]{Opg}.}
  \label{PC}
  Let  $S(t) = \sum_{i\geq 0} c_i t^i$  be a formal power series.
  Suppose that for some positive integer $p$,
  the expression $\sum_{i=0}^{pn-1} c_i$ is a polynomial $P_p(n)$ in the
  variable $n$.  Then the constant term of $P_p(n)$ is independent of $p$.
  We call this constant term the \emph{periodic constant} of $S$ and
  denote it by $\mathrm{pc}(S)$.
\end{definition}

\bekezd\label{be:EEE}{\bf The surgery formula for $\frs$.} \
The relation between the coefficients $h^u_a$ defined in (\ref{ss:first}) and the
series $\cH_{[l'],u}(t)$ is realized as follows
(below, $\{r\}$ will denote the fractional part of the rational number $r$):
\begin{theorem}\labelpar{prop:hpc}
Consider the graphFix  $l'=\sum_va_vE^*_v=\sum_vl'_vE_v$  $\Gamma$ and let $u$ be one of its end--vertices.
with $a$ as in (\ref{cor:ha}).
Abridge $l'_u$ by $\ell$. Then

\vspace{2mm}

(a) \ If \ $\cH_{[l'],u}(t)= \sum_{i\geq 0} c_i t^i$  \ then \  $h^u_a=\sum_{i<d\ell}\, c_i$.

\vspace{2mm}

(b) \ Take  $\bar{l}'\in L'$ such that   $(\bar{l}',E^*_u)\in (-1,0]$.
Then
$$\mathrm{pc}(\cH_{[\bar{l}'],u})=
-\frs^\Gamma_{[-\bar{l}']}-\frac{(K_\Gamma+2\bar{l}')^2+|\cV|}{8}
+\frs^{\Gamma\setminus u}_{[-R(\bar{l}')]}+\frac{(K_{\Gamma\setminus u}+2R(\bar{l}'))^2+|\cV\setminus u|}{8}.$$
\end{theorem}

\begin{proof}
Write $Z_{[l']}(\bt)=\sum c_{l''}\bt^{l''}$. Then (a) follows from
$\cH_{[l'],u}(t)=\sum  c_{l''}t^{dl''_u}$ and
$$h^u_a=\sum _{b\,:\, n_u<0}q^\Gamma_b=
\sum _{l''\,:\,l''_u<l'_u}c_{l''}.$$
(b) For any fixed $\bar{l}'$ as in the assumption of (b),
take $l'$ such that $[l']=[\bar{l}']$, and $[R(l')]=[R(\bar{l}')]$
and $l'=\sum_va_vE^*_v=\sum_vl'_vE_v$ with $a\gg0$. Since $\bar{l}'-l'\in L$, we get that
$l'_u+(\bar{l}',E^*_u)=(\bar{l}'-l',E^*_u)\in\Z$, hence $-(\bar{l}',E^*_u)$ is the fractional part of $l'_u$.
Write $l'_u=\ell$ and let $n$ be its integral part.
For this $l'$ we apply part (a) and the identity of (\ref{cor:ha}).

In order to compute the periodic constant of $\cH_{[l'],u}(t)$,
notice that if $c_{l''}\not=0$ then $l''-l'\in L$, hence
$l''_u-l'_u=(E^*_u,l'-l'')\in \setZ$. Therefore, if the coefficient $c_i$ of $\cH_{[l'],u}$ is nonzero, then
$i-d\{\ell\}\in d\setZ$. In particular, for $c_i\not=0$, $i<d\ell$ if and only if $i\leq dn-1$. This shows that
one has to write  $h^u_a$ as $P(n) $ for some polynomial $P$, and then, the periodic constant of
$\cH_{[l'],u}$ is $P(0)$. From  (\ref{cor:ha}), we get that $P(n)$ is the right hand side of that identity
after the substitution $(l',E^*_u)=-l'_u=-\{\ell\}-n$. Therefore:
$$\mathrm{pc}(\cH_{[l'],u})=
-\frs^\Gamma_{[-l']}-\frac{K_\Gamma^2+|\cV|}{8}
+\frs^{\Gamma\setminus u}_{[-R(l')]}+\frac{K_{\Gamma\setminus u}^2+|\cV\setminus u|}{8}
-\frac{\{\ell\}\cdot (\{\ell\}-(K_\Gamma,E^*_u))}{2(E^*_u,E^*_u)}.$$
Now, this identity  provides  (b) by a straightforward computation (similar to
the computation from the proof of (\ref{cor:ha})).

Strictly speaking, the above argument proves the surgery formula (b)
only if $\delta_w=2$ (cf. the assumption of (\ref{prop:ad})). In general we argue as follows:
let $u$ be an end--vertex, and let us blow up the unique edge adjacent to $u$ getting in this way $\Gamma^b$.
Then the newly created vertex $w$ has $\delta_w=2$. Hence in this situation we can apply the above proof
for $\Gamma^b$ and $\Gamma^b\setminus u$. Since blowing down the $w$--vertex in $\Gamma^b\setminus u$ we get
$\Gamma\setminus u$, and all the involved invariants in (b) are stable with respect to blow up/down,
the result follows.  Indeed,
$$(K+2l')^2+|\Gamma|=(K^b+2\pi^*(l'))^2+|\Gamma^b|,$$
the $\frs$-terms are stable by  (\ref{th:1})(b), and $\cH_{[l'],u}$
by the fact that it depends only on the numbers $(E^*_{v_1},E^*_{v_2})$ where $\delta_{v_i}\not=2$.
\end{proof}

\subsection{The identification of $\frs$ with the Seiberg--Witten invariant}\labelpar{s:SW}

\bekezd\labelpar{ss:SW}{\bf Some facts about the Seiberg--Witten invariant of $M$.}
Let $\Gamma$ be a connected negative definite plumbing graph, and
let $\widetilde{X}$ be the plumbed 4--manifold constructed from $\Gamma$. If $\Gamma$ is a resolution graph,
e.g. as in  (\ref{ss:11}), then the (diffeomorphism type) of the resolution serves for it.
Let $\widetilde{\sigma}_{can}$ be the {\it canonical $spin^c$--structure} on $\widetilde{X}$; its
first Chern class $c_1( \widetilde{\sigma}_{can})$ is $-K\in L'$, cf. \cite[p.\,415]{GS}.
The set of $spin^c$--structures $\mathrm{Spin}^c(\widetilde{X})$ is an $L'$--torsor; if we denote
the $L'$--action by $l'*\widetilde{\sigma}$, then $c_1(l'*\widetilde{\sigma})=c_1(\widetilde{\sigma})+2l'$.

If $M=M(\Gamma)$ is the plumbed 3--manifold associated with $\Gamma$, then $M=\partial \widetilde{X}$,
and all the $spin^c$--structures of $M$ are obtained by restrictions from $\widetilde{X}$.
$\mathrm{Spin}^c(M)$ is an $L'/L$--torsor, compatible with the restriction and the projection $L'\to L'/L$.
The {\it canonical $spin^c$--structure} $\sigma_{can}$ of $M$ is the restriction of $\widetilde{\sigma}_{can}$.

We denote the Seiberg--Witten invariant by $\frsw:\mathrm{Spin}^c(M)\to\setQ$, $\sigma\mapsto \frsw_\sigma$.
Next we recall a  surgery formula satisfied by them, proved in \cite{BN}.

Let us fix one of the end--vertices $u$ of $\Gamma$ (though the statement is true for any vertex, cf.
[loc.cit]). Let $\widetilde{X}(\Gamma\setminus u)$ be the
tubular neighbourhood of $\cup_{v\not=u}E_v$ in $\widetilde{X}$, and $M_u$ its boundary. Hence,
$M_u$ is the plumbed 3--manifold associated with $\Gamma\setminus u$.

Fix any $\sigma\in \mathrm{Spin}^c(M)$, extend it to a $spin^c$--structure $\widetilde{\sigma}\in \mathrm{Spin}^c(\widetilde{X})$ of the form $\widetilde{\sigma}=\widetilde{l}'*\widetilde{\sigma}_{can}$ with
$\widetilde{l}'\in L'$, $(\widetilde{l}',E^*_u)\in [0,1)$. Then consider $R(\widetilde{l}')\in L'(\Gamma\setminus u)$
and the restriction $\widetilde{\sigma}_u$ and $\sigma_u$ of $\widetilde{\sigma}$ to
$\widetilde{X}(\Gamma\setminus u)$ and
$\partial \widetilde{X}(\Gamma\setminus u)$ respectively. Then the main result of
\cite{BN} says the following identity.
(Here we wish to draw the reader's attention to the notational differences between the present note and
\cite{BN}: in that article the `dual' $E^*_v$ has opposite sign, this creates a sign difference
in $(\widetilde{l}',E^*_u)\in [0,1)$, and also in the expression of $\cH$
from (\ref{eq:Po})  the characters $\rho\in
(L'/L)\,\widehat{}$ \, should be replaced by their inverses. Hence, $\cH_{\sigma,u}$ of \cite{BN}
is our $\cH_{[-\widetilde{l}'],u}$.)

\begin{theorem}\labelpar{th:BN}\cite{BN}
$$\frsw_\sigma(M)+\frac{c_1(\widetilde{\sigma})^2+|\cV|}{8}=-\mathrm{pc}(\cH_{[-\widetilde{l}'],u})+
\frsw_{\sigma_u}(M_u)+\frac{c_1(\widetilde{\sigma_u})^2+|\cV\setminus u|}{8}.$$
\end{theorem}

\bekezd\labelpar{ss:MRa}{\bf The proof of Theorem A} is completed by the next result.
\begin{theorem}\labelpar{th:MR}
Let $\frs:L'/L\to\setQ$ be the invariant defined from the series $Z(\bt)$ in (\ref{eq:s}).
Then for any $[l']\in L'/L$ one has:
$$\frs_{[l']}=\frsw_{[l']*\sigma_{can}}.$$
\end{theorem}
\begin{proof}
Notice that $-\bar{l}'$ considered in (\ref{prop:hpc}) satisfies the needed assumptions for
$\widetilde{l}'$ in (\ref{ss:SW}). Moreover, $c_1(\widetilde{\sigma})=-(2\bar{l}'+K)$, and
$c_1(\widetilde{\sigma}_u)=-R(2\bar{l}'+K)$, and $\sigma=[-\bar{l}']*\sigma_{can}$.
Hence, the surgery identities (\ref{th:BN}) and (\ref{prop:hpc})(b) show that
$\frs_{[-l']}$ satisfy the same surgery formula than $\frsw_{[-l']*\sigma_{can}}$.
One can verify that they coincide for graphs with one vertex (or one can apply
\cite[\S 4]{coho3} which shows that they coincide for links of  splice--quotient
singularities, including all Seifert-manifold). Therefore, by induction on the
number of vertices, we get the result.
\end{proof}

\bekezd\labelpar{ss:MRb}{\bf The proof of Theorem B} is a combination of (\ref{eq:INVOL}),
(\ref{eq:SEU}) and  (\ref{th:MR}).

\end{document}